\documentclass[12pt]{article}
\usepackage{amssymb,amsmath,amsfonts,bbm,pifont,upgreek,bbold,accents,xcolor,appendix}  
\usepackage{biblatex}
\usepackage{hyperref}

\hypersetup{urlcolor=blue, citecolor=red}

\usepackage{amssymb,bbm,pifont,upgreek}
\font\teneufm=eufm10
\font\seveneufm=eufm7
\font\fiveeufm=eufm5
\newfam\eufmfam
\textfont\eufmfam=\teneufm
\scriptfont\eufmfam=\seveneufm
\scriptscriptfont\eufmfam=\fiveeufm

\newcommand\beq[1]{ \begin{equation}\label{#1} }
\newcommand{\eeq}{ \end{equation} }

\newcommand{\beqa}{ \begin{align*} }
\newcommand{\eeqa}{ \end{align*} }

\newcommand{\beqano}{ \begin{eqnarray*} }
\newcommand{\eeqano}{ \end{eqnarray*} }

\newtheorem{theorem}{Theorem}
\newtheorem{definition}{Definition}
\newtheorem{proposition}{Proposition}
\newtheorem{lemma}{Lemma}
\newtheorem{sublemma}{Sublemma}
\newtheorem{remark}{Remark}
\newtheorem{notationalremark}{Notations}
\newtheorem{corollary}{Corollary}
\newtheorem{assumption}{Assumption}
\newtheorem{claim}{Claim}

\newtheorem{question}{Question}
\newtheorem{tools}{$\negsp\negsp$}[subsection]

\newcommand\thm[1]{ \begin{theorem}\label{#1}}
\newcommand\thmtwo[2]{ \begin{theorem}[#1]\label{#2}}
\newcommand\ethm{ \end{theorem} }
\newcommand\dfn[1]{ \begin{definition}\label{#1} \rm}
\newcommand\dfntwo[2]{ \begin{definition}[#1]\label{#2} \rm}
\newcommand\edfn{ \end{definition} }
\newcommand\pro[1]{ \begin{proposition}\label{#1}}
\newcommand\protwo[2]{ \begin{proposition}[#1]\label{#2}}
\newcommand\epro{ \end{proposition} }
\newcommand\lem[1]{ \begin{lemma}\label{#1}}
\newcommand\lemtwo[2]{ \begin{lemma}[#1]\label{#2}}
\newcommand\elem{ \end{lemma} }
\newcommand\sublem[1]{ \begin{sublemma}\label{#1}}
\newcommand\sublemtwo[2]{ \begin{sublemma}[#1]\label{#2}}
\newcommand\esublem{ \end{sublemma} }
\newcommand\rem[1]{ \begin{remark}\label{#1} \rm}
\newcommand\erem{ \end{remark} }
\newcommand\notrem[1]{ \begin{notationalremark}\label{#1} \rm}
\newcommand\enotrem{ \end{notationalremark} }
\newcommand\cor[1]{ \begin{corollary}\label{#1}}
\newcommand\cortwo[2]{ \begin{corollary}[#1]\label{#2}}
\newcommand\ecor{ \end{corollary} }
\newcommand\asmp[1]{ \begin{assumption}\label{#1}}
\newcommand\asmptwo[2]{ \begin{assumption}[#1]\label{#2}}
\newcommand\easmp{ \end{assumption} }
\newcommand\clm[1]{ \begin{claim}\label{#1}}
\newcommand\eclm{ \end{claim} }

\expandafter\chardef\csname pre amssym.def
at\endcsname=\the\catcode`\@
\catcode`\@=11
\def\undefine#1{\let#1\undefined}
\def\newsymbol#1#2#3#4#5{\let\next@\relax
 \ifnum#2=\@ne\let\next@\msafam@\else
 \ifnum#2=\tw@\let\next@\msbfam@\fi\fi
 \mathchardef#1="#3\next@#4#5}
\def\mathhexbox@#1#2#3{\relax
 \ifmmode\mathpalette{}{\m@th\mathchar"#1#2#3}%
 \else\leavevmode\hbox{$\m@th\mathchar"#1#2#3$}\fi}
\def\hexnumber@#1{\ifcase#1 0\or 1\or 2\or 3\or 4\or 5\or 6\or 7\or
8\or
 9\or A\or B\or C\or D\or E\or F\fi}
\ifcase\@ptsize
 \font\tenmsb=msbm10
 \font\sevenmsb=msbm7
 \font\fivemsb=msbm5
\or
 \font\tenmsb=msbm10 scaled \magstephalf
 \font\sevenmsb=msbm7 scaled \magstephalf
 \font\fivemsb=msbm5  scaled \magstephalf
\or
 \font\tenmsb=msbm10 scaled \magstep1
 \font\sevenmsb=msbm7 scaled \magstep1
 \font\fivemsb=msbm5 scaled \magstep1
\fi
\newfam\msbfam
\textfont\msbfam=\tenmsb
\scriptfont\msbfam=\sevenmsb
\scriptscriptfont\msbfam=\fivemsb
\edef\msbfam@{\hexnumber@\msbfam}
\def\Bbb#1{\fam\msbfam\relax#1}
\def\widehat#1{\setboxz@h{$\m@th#1$}%
 \ifdim\wdz@>\tw@ em\mathaccent"0\msbfam@5B{#1}%
 \else\mathaccent"0362{#1}\fi}
\def\widetilde#1{\setboxz@h{$\m@th#1$}%
 \ifdim\wdz@>\tw@ em\mathaccent"0\msbfam@5D{#1}%
 \else\mathaccent"0365{#1}\fi}

\def\RIfM@{\relax\ifmmode}
\def\nonmatherr@#1{\errmessage{\string#1\space allowed only in math mode}}
\def\Bbb{\RIfM@\expandafter\Bbb@\else
 \expandafter\nonmatherr@\expandafter\Bbb\fi}
\def\Bbb@#1{{\Bbb@@{#1}}}
\def\Bbb@@#1{\fam\msbfam\relax#1}
\def\setboxz@h{\setbox\z@\hbox}
\def\wdz@{\wd\z@}
\catcode`\@=\csname pre amssym.def at\endcsname
\newcommand{\noi}{{\noindent}}

\newcommand{\negsp}{\hspace{-.09truecm}}  

\renewcommand{\Im}{{\, \rm Im\, }}

\renewcommand\subset{\subseteq}
\renewcommand\supset{\supseteq}

\title{{\bf{Advances on Stable Ergodicity of Toral Automorphisms}}}
\author{F. Argentieri$^{\dag}$, A. Ulliana$^{\ddag}$
\vspace{2mm}
\\ \small
$^{\dag}$ IMPA, Est. D. Castorina 110, Jardim Botanico, 22460-320 Rio de Janeiro, Brazil
\\ \small 
$^{\ddag}$Institut f\"ur Mathematik, Universit\"at Z\"urich
Winterthurerstrasse 190, CH-8057 Z\"urich, CH
\\ \small 
E-mail: fernando.argentieri@impa.br, andrea.ulliana@math.uzh.ch}
\date{}

\begin{document}

\maketitle

\begin{abstract}
    We prove that all ergodic automorphisms of the $N$-dimensional torus with two dimensional center are stably ergodic. This includes all ergodic automorphisms in dimension $N\leq 5$ or $N=7$.\newline
    This generalizes a previous result of Rodriguez-Hertz, that required an additional algebraic condition on the carachteristic polynomial of the linear automorphism. The core of the proof is a minimality criterion.
\end{abstract}

\section{Introduction}

Any element $A$ of $SL(N,\mathbb{Z})$ induces a linear automorphism of the torus $\mathbb{T}^N=\mathbb{R}^N/\mathbb{Z}^N$, that we will denote again by $A$. These maps have a remarkable richness of dynamical features and for this reason they have often served as playground for the early study of many fundamental aspects of dynamical systems. Among several examples we find: entropy \cite{adler1967entropy}, Markov partitions \cite{adler1970similarity}, the Bernoulli property \cite{katzenlson1971ergodic}, and partial hyperbolicity \cite{lind1982dynamical}.\newline
The question about their ergodicity has long been answered, since 1932 in \cite{halmos1943automorphisms} by Halmos, but the study of their statistical properties is still active today \cite{le1999limit}, \cite{dedecker2012empirical}. By Fourier analysis, it is easy to see that $A$ is ergodic with respect to the Lebesgue measure on $\mathbb{T}^N$ if and only if none of the eigenvalues of $A$ is a root of unity.\newline
This paper deals with stable ergodicity of linear automorphisms of the torus, aiming to advance towards answering the following question, posed by Hirsch, Pugh and Shub in 1977 \cite{hirsch1977invariant}.

\begin{question}\label{question}
    Are all ergodic linear automorphism of $\mathbb{T}^N$ also stably ergodic?
\end{question}
The stability has to be intended in $\mathcal{C}_{\text{vol}}^k(\mathbb{T}^N)$ for some $k$, ideally $1+\alpha$, where $\mathcal{C}_{\text{vol}}^k(\mathbb{T}^N)$ denotes the space of $\mathcal{C}^k$ volume preserving diffeomorphisms of $\mathbb{T}^N$.\newline
While stable ergodicity in the $\mathcal{C}^1$ category would be highly desirable, so far there are no known examples of maps with such property and all known mechanisms to prove ergodicity require higher regularity, even in the simplest settings \cite{MR2736152}.\newline

The characterization of stable ergodicity of diffeomorphisms on compact manifolds has been a central problem in smooth dynamics (see \cite{MR4198639} for an historical account) and it is strictly linked to the degree of hyperbolicity of the said diffeomorphisms. In our setting, $A$ is partially hyperbolic, in fact, one has the $A$-invariant splitting
\begin{equation*}
    \mathbb{R}^N=E^s\oplus E^c \oplus E^u,
\end{equation*}
where $E^s$, $E^c$ and $E^u$ are the sums of the generalized eigenspaces of $A$ corresponding to the eigenvalues of modulus less than 1, exactly $1$ and more than 1, respectively. $E^s$ and $E^u$ are, respectively, uniformly contracted and expanded
by $A$, while vectors in $E^c$ are neither contracted or expanded as strongly as the
vectors in $E^s$ nor in $E^u$. We also remark that, for algebraic reasons, $\dim(E^c)$ must be even.\newline

The first example of stably ergodic diffeomorphisms was given by the class of Anosov diffeomorphisms (i.e. uniformly hyperbolic). Uniform hyperbolicity is itself an $\mathcal{C}^1$-open condition and it implies ergodicity for any $\mathcal{C}^{1+\alpha}$ conservative diffemorphism, as proved by Anosov \cite{ansov1969geodesic} in 1969, exploiting the the Hopf argument from \cite{MR1464}. When $E^c$ is trivial, $A$ is uniformly hyperbolic, settling question \ref{question} in the affirmative in this case.\newline

Partial hyperbolicity is a $\mathcal{C}^1$-open condition, but it does not imply ergodicity alone. In 1994 in \cite{MR1298715}, the first example of non uniformly hyperbolic stably ergodic diffeomorphism was given, exploiting its partial hyperbolicity and an additional property, called accessibility, which, euristically, allows to replicate the Hopf argument. Soon afterwards, stable ergodicity was conjectured to be prevalent ($\mathcal{C}^r$-dense) among partilly hyperbolic diffeomorphisms by Pugh and Shub \cite{MR1449765}, proposing accessibility as a main tool.\newline
Indeed, Burns and Wilkinson \cite{burns2010ergodicity} proved that essential accessibility (a weakening of accessibility) implies ergodicity of $\mathcal{C}^2$ conservative diffeomorphisms under mild additional assumptions (center bunching). Moreover, accessibility was proved to be $\mathcal{C}^1$-stable when the center is low dimensional \cite{MR4142463}, \cite{DIDIER_2003}, and stable accessibility to be $\mathcal{C}^1$-dense \cite{MR2039999}. However, because of the missmatch between the $\mathcal{C}^{1+\alpha}$ requirement for ergodicity and the only $\mathcal{C}^1$-prevalence of accessibility this was not enough to complete Pugh and Shub program, which remains widely open today. The best results up to date exploit a different approach, based on blenders \cite{MR4198639}.\newline

Out of completeness, we mention that there are examples of non partially hyperbolic diffeomorphisms that are stably ergodic \cite{MR2085722}, but every stably ergodic diffeomorphism must have a weak form of hyperbolicity. In fact, the existence of a dominated splitting is a necessary condition for stable ergodicity \cite{MR2018925}.\newline

Coming back to our setting, when $E^c$ is non trivial, $A$ is partially hyperbolic but not accessible, preventing us from exploiting the afore mentioned results.
At least, the ergodicity of $A$ ensures the essential accessibility property, which, though, is not necessarily a stable property, posing a serious challenge in establish stable ergodicity results.\newline

A major breakthrogh was achieved by Rodriguez-Hertz \cite{Hertz2005}, who proved that $A$ is actually stably essentially accessible in the $\mathcal{C}^{22}$-topology under the assumptions that $\dim(E^c)=2$ and of $A$ being pseudo-Anosov (see section \ref{linear algebra} for the definition). The latter is an algebraic condition of the carachteristic polynomial of $A$, which controls the action of $A$ on $\mathbb{Z}^N$. This class of linear automorphisms includes all non Anosov ones in dimension 4, but, on the other hand, is empty in odd dimension.\newline

Our result removes any algebraic constraint on $A$.

\begin{theorem}\label{stable ergodicity}
    Any ergodic linear automorphism $A$ of the torus $\mathbb{T}^N$ with $\dim(E^c)=2$ is stably ergodic in $\mathcal{C}^{22}_{\text{vol}}$.
\end{theorem}
We remark that in any dimension $N\geq 6$ there exist linear automorphisms of $\mathbb{T}^N$ that have $\dim(E^c)=2$ but are not pseudo-Anosov. This is easily seen, thanks to the fact that any integer coefficient polynomial is the carachteristic polynomial of a suitable element of $SL(N,\mathbb{Z})$.\newline

In dimension $7$ the ergodicity of $A$ implies that $\dim(E^c)=0$ or $\dim(E^c)=2$ (see lemma \ref{dimension 7 polynomial}), giving us the following corollary directly.

\begin{corollary}\label{dimension 7}
    Any ergodic linear automorphism $A$ of $\mathbb{T}^7$ is stably ergodic in $\mathcal{C}^{22}_{\text{vol}}$.
\end{corollary}

\begin{remark}
    The same conclusion holds in dimension $6$ and $9$, provided that no eigenvalue of $A$ is a Salem number.
\end{remark}

Finally, we want to comment about the remaining assumptions of our result.\newline
As in \cite{Hertz2005}, the high regularity of the perturbations is reuqired specifically to be able to exploit KAM theory (see \cite{MR538680} for an exposition) in the final step of the proof. There is no indication for this to be a necessary condition, but a new general strategy of proof would be required to reduce such regularity requirement.\newline
The condition that $\dim(E^c)=2$ is essential for understanding the structure of the accessibile components (called accessibility classes) of $f$ on $\mathbb{T}^N$. They are, in general, conjectured to be topological manifolds (smooth manifolds under center bunching) \cite{MR2391330},it is proved only for low dimensional center \cite{rodriguez2017structure}. This assumption is also used in the algebraic topological argument and to control the number of possible dimensions for the accessibility classes.

\subsection*{Acknowledgments}

The authors thank their mentor Artur Avila for drawing their attention to this topic, Federico Rodriguez Hertz for the valuable conversations, and Daniele Galli for the encouragements.\newline
F.A. was supported by the SNSF mobility grant.


\section{Background and Organization}

\subsection{Partial Hyperbolicity}

We recall some general theory about partial hyperbolicity.\newline
Let $M$ be a Riemmanian manifold endowed with its volume measure and let $f:M\rightarrow M$ be a diffeomorphism. We say that $f$ is partially hyperbolic if there exists a $Df$-invariant splitting
\begin{equation*}
    TM = E_f^s \oplus E_f^c\oplus E_f^u,
\end{equation*}
where $E_f^s$ and $E_f^u$ are non trivial bundles, and for $*=s,c,u$ there exist continuous functions $\mu_*,\lambda_*:M\rightarrow\mathbb{R}_+$ such that
\begin{equation*}
    \mu_*(x) < \frac{\|D_xf(v^*)\|}{\|v\|} < \lambda_*(x), \quad \text {for every } v^*\in E_f^*(x)\setminus\{0\},
\end{equation*}
and
\begin{equation*}
    \lambda_s(x) < \min\{1,\mu_c(x)\}, \quad \quad \max\{1,\lambda_c(x)\}<\mu_u(x).
\end{equation*}
In other words, $E^s_f$ and $E^u_f$ are hyperbolic bundles, and the vectors of $E^c_f$ are contracted or expanded less than the ones of $E^s_f$ and $E^u_f$.\newline
If $\mu_*$ and $\lambda_*$ do not depend on $x$, for $*=s,c,u$, then we say that $f$ is absolutely partially hyperbolic \cite{brin1974partially}. In addition, we say that $f$ is center bunching if
\begin{equation*}
    \lambda_s(x) < \frac {\|D_xf(v)\|}{\|D_xf(w)\|} <\mu_u(x) \quad \text{for any } v,w\in E^c_f(x)\setminus\{0\}.
\end{equation*}

We always have that $E^s_f$ and $E^u_f$ are uniquely integrable to invariant Holder continuous foliations $\mathcal{F}_f^s$ and $\mathcal{F}_f^u$. If the distributions $E^c_f$, $E_f^{cs}=E^s_f\oplus E^c_f$ and $E_f^{cu}=E^u_f\oplus E^c_f$ are uniquely integrable to foliations $\mathcal{F}_f^c$, $\mathcal{F}_f^{cs}$ and $\mathcal{F}_f^{cu}$, we say that $f$ is dynamically coherent. We denote by $W_f^*(x)$ the leaf of $\mathcal{F}_f^*$ that contains $x$, whenever this foliation exists and is unique, for $*=s,c,u,cs,cu$. Finally, dynamical coherence is a $\mathcal{C}^1$-stable property, provided that $f$ is absolutely partially hyperbolic (see \cite{hirsch1977invariant} or \cite{pesin2004lectures}).\newline

We say that a curve $\gamma:I\rightarrow M$ is an $su$-path, if it is piecewise contained in the leaves of $\mathcal{F}^s$ and $\mathcal{F}^u$. We say that the minimal number of such pieces is the number of legs of $\gamma$. A point $y$ is said to be $su$-accessible from $x$, if there exists a $su$-path joining the two and this is an quivalence relation. We denote the equivalence class of $x$ by $AC(x)=\{y\in M: y \text{ is accessible from } x\}$ and we refer to it as the accessibility class of $x$. A set $E\subset M$ is said $su$-saturated if, for every $x\in E$, we have that $AC(x)\subset E$. We say that $f$ is accessible if $AC(x)=M$ for every $x\in M$. Instead, $f$ is said essentially accessible (with respect to the volume) if every measuable $su$-saturated $E\subset M$ has either full or zero volume.\newline

As mentioned in the introduction essential accessibility is one of the main tools to prove ergodicity of partially hyperbolic systems with respect to the volume measure. We recall Burns-Wilkinson theorem \cite{burns2010ergodicity}.

\begin{theorem}
    Let $f:M\rightarrow M$ be a $\mathcal{C}^2$ volume-preserving, partially hyperbolic and center bunched diffeomorphism on a closed Riemmanian manifold $M$. If $f$ is essentially accessible, then $f$ is ergodic and in fact has the Kolmogorov property.
\end{theorem}

\subsection{Lifts and Perturbations}

In the rest of the paper $f:\mathbb{T}^N\rightarrow \mathbb{T}^N$ will denote a perturbation of $A$ in $\mathcal{C}^k_{vol}(\mathbb{T}^N)$.\newline
First of all, we notice that $0\in\mathbb{T}^N$ is a fixed point for $A$. Given the ergodicity of $A$, none of its eigenvalues can be equal to 1, implying that this fixed point is stable under small $\mathcal{C}^1$ perturbations, i.e., up to a small translational change of variable, we can assume that $f(0)=0$.\newline

We will often work in $\mathbb{R}^N$, the universal cover of $\mathbb{T}^N$, denoting the projection to the torus by $\pi:\mathbb{R}^N\rightarrow\mathbb{T}^N$. We consider the unique lift $F:\mathbb{R}^N\rightarrow\mathbb{R}^N$ of $f$ that fixes $0$. Finally, we can write $F = A + G$ with $G$ being $\mathbb{Z}^N$-periodic.\newline

We recall the splitting $\mathbb{R}^N=E^s\oplus E^c \oplus E^u$, mentioned in the introduction. Given $v\in\mathbb{R}^N$, we denote the corresponding decomposition by $v = v^s + v^c + v^u$. We also set $E^{cs}=E^c\oplus E^s$, $E^{cu}=E^c\oplus E^u$, $E^{su}=E^s\oplus E^u$ and, analogously, $v^{cs}=v^s+ v^c$, $v^{cu}= v^c + v^u$, $v^{su}=v^s+ v^u$. Finally, we endow $E^s,E^u$ and $E^c$ with norms making $A|_{E^s}$ and
$A^{-1} |_{E^u}$ contractions and $A|_{E^c}$ an isometry. On $\mathbb{R}^N$ we set the norm $|v|=|v^u|+|v^c|+|v^s|$. 
In this way we have that $(E^s)^{\perp}=E^{cu}$, $(E^c)^{\perp}=E^{su}$, $(E^u)^{\perp}=E^{cs}$.\newline 

It is easy to check that the constant bundles $E^*_A(x)=E^*$ for $*=s,c,u$, provide an $A$-invariant splitting of $T\mathbb{T}^N$, which is absolutely partially hyperbolic. Being these distributions constant, $A$ is dynamically coherent. In view of the previous discussion, provided that $f$ is sufficiently $\mathcal{C}^1$-close to $A$, $f$ and $F$ are dynamically coherent partially hyperbolic diffeomorphisms.\newline

We denote by $W^*_f$ and $W^*_F$ their invariant manifolds, for $*=s,c,u,cs,cu$. Very often we will omit the index, since the environment manifold ($\mathbb{T}^N$ or $\mathbb{R}^N$) will distinguish between the two, and we will use a different notation for the invariant manifolds of $A$.\newline

\subsection{Plan of Proof and Structure of the Paper}

Here we describe the general strategy of proof of Theorem \ref {stable ergodicity}, which is similar to the one in \cite{Hertz2005}. In this setting, the accessibility classes of $f$ are injectively immersed manifolds \cite{rodriguez2017structure} and their dimensions vary semi-continuously (section \ref{proof of ergodicity}). It is an open problem to determine if this holds in general \cite{MR2391330}.The first step in the proof is to prove that all accessibility classes of $f$ have the same dimension and in view of the semi-continuity, the following minimality criterion is all we need.

\begin{theorem}\label{minimality} Let $A$ be a linear automorphism of $\mathbb{T}^N$ with $\dim(E)^c=2$.\newline
    If $f$ is sufficiently close to $A$ in $\mathcal{C}_{\text{vol}}^1(\mathbb{T}^N)$, for every $E\subset\mathbb{T}^n$ $su$-saturated, $f$-invariant non-empty closed set, $E=\mathbb{T}^N$.
\end{theorem} 
 
To conclude the proof we can now reason by cases. If the codimension of the classes is 0, they must be open and therefore $f$ is accessible by connectedness. Codimension 1 is ruled out by analyzing the fixed point of $f$ (namely $0$). Finally, if the codimension is 2, one can deduce the joint integrability of stable and unstable foliations. In this situation, exploiting some KAM theory, it is possible to $\mathcal{C}^1$-conjugate these foliation to those of $A$, deducing that $f$ is essentially accessible. Either way, ergodicity is finally obtained thanks to Burns-Wilkinson theorem, establishing Theorem \ref{stable ergodicity}.\newline

The novelty of our work stands in the generality of the minimality criterion (Theorem \ref{minimality}). Rodriguez-Hertz exploits a version which requires $A$ to be pseudo-Anosov, and it is the only part of his work where such assumption appears. The core of this paper is devoted to the proof of this theorem, which we sketch here.\newline

We will work in $\mathbb{R}^N$, considering $F:\mathbb{R}^N\rightarrow\mathbb{R}^N$, the lift of $f$, and $V=\pi^{-1}(E)$.\newline
The first step is to build a subspace $X\subset\mathbb{R}^N$, that contains $E^c$ and where a suitable power of $A$, let us say $A^k$, acts in a pseudo-Anosov way. This is section \ref{linear algebra}.\newline
We remark that this is not enough to reduce ourselves to Rodriguez-Hertz's argument. In fact, $X$ is typically not invariant under $F$, nor under $F^k$.\newline
We then examine a connected component $U$ of $V$, aiming to prove that $U=\mathbb{R}^N$.\newline
By a volume argument, valid for any $su$-satuared open set, we obtain the esistence of $n\in X\cap \mathbb{Z}^N$ such that $U+n=U$. By the algebraic properties of $A^k|_X$, we can extend the transaltional invariance of $U$ to a lattice $\Gamma$ of $X$ (subsection \ref{the structure of U}).\newline
It is also possible to show that $U$ is simply connected (appendix \ref{Simply connected proof}).\newline
By a topological argument, any simply connected set, which is invariant along directions that span $E^c$, must intersect every accessibility class (subsection \ref{main proof}). The $su$-saturation of $U$ let us conclude.\newline

We remark that theorem \ref{minimality}, apart from being the key of our argument, has a relevance on its own. The assumption that $\dim(E^c)=2$ is not restrictive in low dimension and small perturbation of linear automorphisms of the torus are commonly studied maps (see for example \cite{micena2021lyapunov}, \cite{gogolev2025smooth}, \cite{brown2026lyapunov}, \cite{MR4422214}, \cite{Adam_2017}).\newline

Finally, we describe the remaining sections of the paper. In section \ref{geometric preliminaries} we list some properties of the invariant manifolds and the holonomy maps they induce, which are peculiar to this setting. Section \ref{linear algebra} is dedicated to the algebraic preliminaries. In section \ref{Proof of the Minimality Criterion} we prove the minimality criterion (theorem \ref{minimality}) and finally, section \ref{proof of ergodicity} proves that the minimality criterion implies our main result (theorem \ref{stable ergodicity}).

\section{Geometric Preliminaries}\label{geometric preliminaries}

\subsection{Invariant Manifolds and Holonomies}\label{Invariant Manifolds and Holonomies}

In this subsection we collect some properties of the invariant manifolds and the associated holonomies, that are peculiar of our setting. These results are directly taken from \cite{Hertz2005}.\newline

In this setting, the invariant manifolds $W_F^*$ are graphs over the subspaces $E^*$. More precisely, there exist $g^*:\mathbb{R}^N\times E^*\rightarrow (E^{*})^{\perp}$ such that $W^*(x) = x+ \text{graph}(g^*(x,\cdot))$. Therefore we have that the parametrization $\sigma^{*}:\mathbb{R}^N\times E^{*} \rightarrow \mathbb{R}^N$ of the invariant foliation $\mathcal{F}^*$, given by $\sigma^{*}(x,v) = x + v + g^{*}(x,v)$. We also denote by $\sigma_x^*:E^*\rightarrow\mathbb{R}^N$ the parametrization of $W^*(x)$, defined by $\sigma_x^*(v)=\sigma^*(x,v)$. \newline

\begin{lemma}\label{invariant manifolds are Lip graphs}
There exist $\kappa=\kappa(f)$, with $\kappa(f)\rightarrow0$ as $f\rightarrow A$ in $\mathcal{C}^1$,  such that, for $*=s,c,u,cs,cu$,
\begin{equation*}
    |g^{*}(x,v)| \leq \kappa|v|, \quad\text{for every }v\in E^{*}.
\end{equation*}
Moreover, the parametrizations $\sigma_x^*$ are $(1+\kappa)$-biLipschitz continuous.
\end{lemma}

\begin{lemma}\label{unique intersection}
    For any $x,y\in\mathbb{R}^N$
    \begin{equation*}
        \#W^s(x)\cap W^{cu}(y)=1, \quad\text{and}\quad \#W^u(x)\cap W^{cs}(y)=1.
    \end{equation*}
\end{lemma}



\noi For $x\in\mathbb{R}^N$ and $y\in W^{cu}(x)$ we have the holonomy $h^u_{xy}:W^c(x)\rightarrow W^c(y)$, defined by 
$$ h^u_{xy}(z) = W^u(z)\cap W^{cs}(y). $$
For $y\in W^{cs}(x)$ one defines $h^s_{xy}:W^c(x)\rightarrow W^c(y)$ by $h^s_{xy}(z)=W^s(z)\cap W^{cu}(y)$.\newline
If $\gamma$ is a $su$-path connecting $x$ and $y$ we can define $h^{\gamma}:W^c(x)\rightarrow W^c(y)$, by concatenating the the stable and unstable holonomies along the corresponding legs of the path.\newline
Given $x,y\in\mathbb{R}^N$ we have a way to define a preferential holonomy between $W^c(x)$ and $W^c(y)$. Exploiting the uniqueness of the intersection between $W^{s}$ and $W^{cu}$, we define $h^{su}_{x,y}=h^{s}_{x,z}\circ h^{u}_{z,y}$, where $z=W^s(x)\cap W^{cu}(y)$.\newline

Exploiting Lemma \ref{unique intersection} we can define
$\pi^s:\mathbb{R}^N\rightarrow W^{cu}(0)$, by $\pi^s(x)=W^s(x)\cap W^{cu}(0)$  and $\pi^u:\mathbb{R}^N\rightarrow W^{cs}(0)$, by $\pi^u(x)=W^u(x)\cap W^{cs}(0)$. Finally we set
\begin{equation*}
    \pi^{su}:\mathbb{R}^N\rightarrow W^c(0) \quad \text{ by } \quad \pi^{su}=\pi^s\circ\pi^u.
\end{equation*}
Notice that $\pi^{su}$ preserves the accessibility classes, i.e. $\pi^{su}(z)\in AC(z), \forall z\in\mathbb{R}^N$.\newline

Then we set $\hat{T}_n :W^c(0)\rightarrow W^c(0)$ by $\hat{T}_n=h^{su}_{n,0}\circ(id+n)$. Notice that, when projected to $\mathbb{T}^N$, these are holonomies along $su$-loops. To make some properties of this family of maps more evident it will be at times useful to read them in charts, hence defining $T_n:E^c\rightarrow E^c$ by $T_n=(\sigma_0^c)^{-1}\circ \tilde T _n \circ \sigma_0^c$.\newline
Notice that, a priori, it might hold that $T_n\circ T_m\neq T_{n+m}$.

\begin{lemma}\label{Holonomies Lip}
    There exists $\beta=\beta(f)$ with $\beta(f)\rightarrow 0$ as $f\rightarrow A$ in $\mathcal{C}^1$ and $C>0$, that depends only on the size of the $\mathcal{C}^1$ neighbourhood of $A$ such that the following holds.\newline
    For any $su$-path $\gamma$ with at most $K$ legs and length less than $L\geq1$,
    \begin{equation*}
        \text{Lip}(h^{\gamma})\leq C^KL^{K\beta}.
    \end{equation*}
    Moreover, for every $n\in\mathbb{Z}^N$ with $n\neq 0$,
    \begin{equation*}
        \text{Lip}(\hat T_n)\leq C|n|^{\beta}.
    \end{equation*}
\end{lemma}

\begin{lemma}\label{Tn log}
    There exists $C >0$ that only depends on the $\mathcal{C}^1$ size of the neighborhood of $A$ such that, for any $x\in E^c$ and $n\in\mathbb{Z}^N$,
    \begin{equation*}
        |T_n(x)-(x+n^c)| \leq C\log^+(|n|)+C.
    \end{equation*}
\end{lemma}
Here $\log^+(x)=\max\{0,\log(x)\}$.

\begin{lemma}\label{Tn close to linear}
If $f$ is $\mathcal{C}^r$, then $T_n$ is $\mathcal{C}^r$ for all $n\in\mathbb{Z}^N$, moreover, writing
$T_n(z) = z+ n^c + \phi_n(z)$
then for any $\varepsilon>0$ and $R>0$ there is a neighborhood of $A$ in the $\mathcal{C}^r$ topology such that if $f$
is in this neighborhood, then $\|\phi _n\|_{\mathcal{C}^r} <\varepsilon$ whenever $|n| \leq R$.
\end{lemma}

\subsection{Saturation}\label{Su-saturation}

Here we prove futher results about the invariant foliations, in particular how they jointly saturate the environment space.\newline

\noi We set $E^*(x)=x+E^*$. Moreover, for $S\subset\mathbb{R}^n$, we define $E^*(S)=\bigcup_{x\in S} E^*(x)$.\newline
We also define define $W_r^{*}(x)=\sigma_x^*(B_{E^*}(0,r))$ and, similarly, we set $W^*(S)=\bigcup_{x\in S} W^*(x)$ and $W_r^*(S)=\bigcup_{x\in S} W_r^*(x)$.\newline

We define $\Phi_x(V):\mathbb{R}^N\rightarrow \mathbb{R}^N$ by $\Phi_x(v) = \sigma^u(v^u,\sigma^s(v^s,\sigma^c(v^c,x)))$.

\begin{lemma}\label{Phi close to identity}
    There exist $\kappa=\kappa(f)$ with $\kappa\rightarrow0$ when $f\rightarrow A$ in $\mathcal{C}^1$ such that
    $$|\Phi_x(v)-(x+v)|\leq\kappa|v|.$$
    Moreover, whenever $f$ is sufficiently $\mathcal{C}^1$-close to $A$, $\Phi_x$ is proper.
\end{lemma}

\begin{proof}
    We have, thanks to lemma \ref{invariant manifolds are Lip graphs}, that
    \begin{align*}
        |\Phi_x(v)-(x+v)|=&|\sigma^u(v^u,\sigma^s(v^s,\sigma^c(v^c,x)))-(x+v)|\\
        \leq&|\sigma^u(v^u,\sigma^s(v^s,\sigma^c(v^c,x))) - (\sigma^s(v^s,\sigma^c(v^c,x)) + v^u)|\\
        &\quad+|\sigma^s(v^s,\sigma^c(v^c,x))-(\sigma^c(v^c,x)+v^s)| +|\sigma^c(v^c,x)-(x+v^c)|\\
        \leq& \kappa|v^u| + \kappa|v^s| + \kappa|v^c|\\
        =& \kappa|v|,
    \end{align*}
    where $\kappa$ is, of course, the same as in Lemma lemma \ref{invariant manifolds are Lip graphs}.\newline
    Finally, we have that $|\Phi_x(v)|\geq (1-\kappa)|v-x|$, hence $\Phi_x$ is proper, whenever $\kappa<1$.
\end{proof}

\begin{lemma}\label{proper 1}
    If $f$ is sufficiently close to $A$ in the $\mathcal{C}^1$ topology we have that $$W_r^u(W_r^s(W_r^c(x)))\supset B(x,r/2).$$ 
    Moreover $\Phi_x$ is surjective.
\end{lemma}

\begin{proof}
    Suppose that $f$ is sufficiently $\mathcal{C}^1$-close to $A$ to ensure that $\kappa$ (from lemma \ref{invariant manifolds are Lip graphs} and \ref{proper 1}) is less than $1/2$. We will prove that $\Phi_x(B(0,r))\supset B(x,r/2)$ and both conclusions will follow directly from this claim.\newline

    We define $H:\overline{B(0,r)}\times[0,1]\rightarrow\mathbb{R}^N$ by $H(v,t) = (1-t)\Phi_x + t(x+v)$. This is the linear interpolation homotopy between $H_0=\Phi_x$ and $H_1=Id+x$.\newline
    For $v\in\partial B(0,r)$, lemma \ref{proper 1} guarantees that $\Phi(x)\in B(x+v,r/2)$, and, by convexity, it follows that $H_t(v)\in B(x+v,r/2)$ for every $t$. Therefore, $H_t(\partial B(0,r))\cap B(x,r/2)=\emptyset$ for every $t$ and, consequently, for every $y\in B(x,r/2)$
    \begin{equation*}
        \deg(\Phi_x,B(0,r),y) = \deg(Id+x,B(0,r),y) = 1.
    \end{equation*}
    This implies that $y\in \Phi_x(B(0,r))$, which gives the conclusion.

\end{proof}

\begin{lemma}\label{homomorphism}
    If $f$ is sufficiently close to $A$ in the $\mathcal{C}^1$ topology, $\Phi_x$ is an homeomorphism and
    \begin{equation*}
        |\Phi^{-1}_x(v)-(v-x)|\leq\frac{\kappa}{1-\kappa}|v-x|.
    \end{equation*}
\end{lemma}

\begin{proof}
    We require that $f$ is in a $\mathcal{C}^1$ neighbourhood of $A$ where $\kappa<1$.\newline

    We start by proving that $\Phi_x$ is injective.\newline
    Suppose that $\sigma^u(v_i^u,\sigma^s(v_i^s,\sigma^c(v_i^c,x)))=z$ for some values of $v_i^*\in E^*$, for $i=1,2$ and $*=s,u,c$. Since 
    $\sigma^s(v_i^s,\sigma^c(v_i^c,x))\in W^u(z)\cap W^{cs}(x)$, by lemma \ref{unique intersection}, it must be the same point for $i=1,2$. We denote this point by $w$. Now, $\sigma^c(v_i^c,x)\in W^s(w)\cap W^{cu}x$, hence it must be again a uniquely determined point $y$. Finally, the injectivity of $\sigma_x^c$, $\sigma_y^s$, $\sigma_z^u$ inplies that $v_1^*=v_2^*$ for $*=c,s,u$, proving the claim.\newline

    This and Lemma \ref{proper 1} prove that $\Phi_x$ is bijective and therefore $\Phi_x^{-1}$ is well-defined.\newline
    We now concern ourselves with the final estimate, leaving the continuity of $\Phi_x^{-1}$ to the end of the proof. From lemma \ref{Phi close to identity}, as long as $\kappa<1$, we deduce that $|\Phi_x(v)-x|>(1-k)|v|$. Chaning this with lemma \ref{Phi close to identity} itself, we get the following.

    \begin{align*}
        |\Phi_x(v) - (v+x)| < \kappa |v| &< \frac {\kappa}{1-\kappa} |\Phi_x(v) - x|,\\
        |\Phi_x^{-1}(w) - (w-x)| &< \frac{\kappa}{1-\kappa}|w-x|,
    \end{align*}
    where the second line follows from the first one with the substitution $v=\Phi_x^{-1}(w)$.\newline

    Finally, recalling that $\Phi_x:\mathbb{R}^N\rightarrow\mathbb{R}^N$ is bijective, continuos and proper, the continuity of $\Phi_x^{-1}$ follows from the next general topology fact.
\end{proof}

\begin{lemma}
Let $X$, $Y$ be Hausdorff topological spaces, with $Y$ admitting a locally finite compact cover. Let $f:X\rightarrow Y$ a continuous proper bijective function.\newline
Then $f$ is an homeomorphism.
\end{lemma}

\begin{proof}
    Let $(K_i)_i$ the locally finite compact cover of $Y$. Since $f$ is proper, $f^{-1}(K_i)$ is compact and therefore $f|_{f^{-1}(K_i)}$ is closed. This, together with $f^{-1}(K_i)$ being closed, implies that $f^{-1}|_{K_i}$ is continuous. We can finally conclude that $f^{-1}$ is continuous, since it is continuous when restricted to all elements of a locally finite closed cover.
\end{proof}\newline

Defining $\Psi_x(v)=\sigma(v^u,\sigma(v^s,x))+v^c$, we have that Lemma \ref{Phi close to identity} holds for $\Psi_x$ and, arguing as in lemma \ref{proper 1}, we get the following.

\begin{lemma}\label{proper 2}
    If $f$ is sufficiently close to $A$ in the $\mathcal{C}^1$ topology, then $$E^c(W^u_r(W^s_r(x)))\supset B(x,r/2)$$
\end{lemma}

\section{Algebraic preliminaries}\label{linear algebra}

In this section we develop all needed algebraic tools. In particular we prove the fundamental Lemma \ref{Pseudo Anosov subspace}, about the action of $A$ on $\mathbb{Z}^N$. Secondly, we classify the possible dimensions of $E^c$ in low dimensions.\newline
Finally we recall Diophantine estimates for vectors in $E^c$, that will be crucial in the conclusion of the proof of stable ergodicity, that exploits some KAM theorems.

\subsection{Characteristic Polynomial Properties}

In this section $V$ will denote a real finite dimesional vector space and $L:V\rightarrow V$ a linear map. Furthermore, we will assume the existence of a maximal rank discrete subgroup $S$ of $V$, which is also $L$-invariant.\newline

\begin{remark}\label{change of basis}
    There always exists a linear isomorphism between $(V,S)$ and $(\mathbb{R}^d,\mathbb{Z}^d)$, where $d=\dim(V)=\text{rank}(S)$. This isomorphism is always bi-Lipschiz, for any two norms on $V$ and $\mathbb{R}^d$.
\end{remark}

In view of this remark, in a suitable basis, $L$ can be represented by a matrix in $SL(d,\mathbb{Z})$. In particular, its characteristic polynomial has always integer coefficients.\newline

We say that a vector $v\in V$ is cyclic for $L$ if $\text{span}_{\mathbb{R}}\{v, Lv, \dots,L^{d-1} v\} = V$.

\begin{lemma}
    The characteristic polynomial $p_L$ of $L$ is irreducible in $\mathbb{Z}[x]$ if and only if every non-zero element of $S$ is cyclic for $L$.
\end{lemma}

\begin{proof}\label{cyclic iff irreducible}
    Thanks to remark \ref{change of basis}, we can assume that $(V,S)=(\mathbb{R}^d,\mathbb{Z}^d)$.\newline
    We prove the first implication.\newline
    For the sake of contradiction, suppose there exists $n\in\mathbb{Z}^d\setminus\{0\}$ and $k<d$ such that we can write $\sum_{i=0}^k m_iA^in=0$, for a non-zero choice of integer coefficients $m_i$.\newline 
    Set $q(x)=\sum_{i=0}^k m_ix^i$. Since $p_A(x)$ is irreducible over $\mathbb{Z}$, it is over $\mathbb{Q}$ and therefore $\mathbb{Q}[x]/(p_A(x))$ is a field. Since $q(x)$ is a non-zero element, there exists $r(x)\in\mathbb{Q}[x]/(p_A)$ such that $r(x)q(x)=1$. Evaluating this expression in $A$, we get $r(A)q(A)=\text{id} + k(A)p_A(A)=\text{id}$ for some $k(x)\in\mathbb{Q}(x)$. Since $p_A(A)$, $q(A)$ has to be an invertible matrix, contraddicting $\sum_{i=0}^k m_iA^in=0$ for $n\neq 0$.\newline
    For the reverse implication, suppose that $P_L$ has a non trivial irreducible factor $q$ in $\mathbb{Z}[x]$. Then, working in $\mathbb{Q}^d$, $\ker{q(L)}$ is a non-trivial $L$-invariant $\mathbb{Q}$-subspace, that contains integer vectors which cannot be cyclic.
\end{proof}

\begin{definition}
    We say that $L$ is Pseudo-Anosov on $(V,L)$ if one of the following three equivalent conditions holds:
    \begin{enumerate}
        \item Every non-zero $s\in S$ is cyclic for $L^k$, for every $k>0$;
        \item The characteristic polynomial $p_{L^k}$ of $L^k$ is irreducible in $\mathbb{Z}[x]$, for every $k>0$;
        \item The characteristic polynomial $p_{L}$ of $L$ is irreducible in $\mathbb{Z}[x]$, and it is not a polynomial in $x^m$ for any $m>1$.
    \end{enumerate}
\end{definition}

The first two conditions are equivalent because of Lemma \ref{cyclic iff irreducible} and are the only two we are going to work with in this paper. The third was included to facilitate comparison with \cite{Hertz2005}, as that is their original definition. The equivalence between that and the second condition is, in fact, proved in \cite{Hertz2005} (Lemma A.9).

\begin{lemma}\label{Pseudo Anosov subspace}
    Let $A\in SL(N,\mathbb{Z})$ have no root of unity as eigenvalue and $\dim(E^c)=2$. Then there exists $k\in\mathbb{N}$, and $X$ real subspace of $\mathbb{R}^N$, such that:
    \begin{itemize}
        \item $E^c\subset X$;
        \item $X$ is $A^k$-invariant;
        \item $\Lambda=\mathbb{Z}^d\cap X$ has full rank in $X$;
        \item $A^k$ is Pseudo-Anosov on $(X,\Lambda)$.
    \end{itemize}
\end{lemma}

\begin{proof}
    Let $p_k(x)$ be the irreducible factor over $\mathbb{Q}$ of the charecteristic polynomial $p_{A^k}(x)$ of $A^k$ that has both unitary roots of $p_{A^k}(x)$. (If it has one it has to have the other one, since it has real coefficients and the two roots are complex conjugates).\newline
    Let us set $d_k=\deg(p_k(x))$.\newline
    We now work in $\mathbb{Q}^N$. Let us set $X_k=\ker(p_k(A^k))$. $X_k$ is an $A^k$-invariant space and the minimal polynomial of $A^k|_{X_k}$ is $p_k(x)$. It has to be the charcacteristic polynomial as well since it must divide $p_{A^k}(x)$ and $p_k(x)$ is an irreducible factor of multiplicity one (since its complex unitary roots have multiplicity one). It follows that $A^k|_{X_k}$ has dimension equal to $d_k$ and irreducible characteristic polynomial over $\mathbb{Z}$. In view of the previous lemma, every non-zero vector of $X_k$ is then cyclic with respect to $A^k$ (considering a suitable integer multiple). This implies that $X_k\cap\mathbb{Z}^N$ has full rank and that $X_k$ has no non-trivial $A^k$-invariant subspaces (and we stress that we are talking about $\mathbb{Q}$-subspaces).\newline
    Set $\tilde{X_k}=\text{span}_{\mathbb{R}}(X_k)$. We have $\tilde{X_k}=\ker(p_k(A^k))$, this time seen as an linear endomorphism of $\mathbb{R}^d$. Let us denote by $q_k(x)$ the real coefficient polynomial that has as roots the unitary roots of $p_{A^k}(x)$. Since $q_k(x)|p_k(x)$, we have that $E^c=\ker(q_k(A^k))\subset \ker(p_k(A^k))=\tilde{X_k}$.\newline
    Since $d_k$ are natural numbers we can pick $k$ such that $d_k$ is minimized. We fix such $k$ for the rest of the proof and set $X=\tilde{X_k}$.\newline
    We claim that $X_{kl}=X_k$ for any $l\geq 1$.\newline
    Since $\tilde{X_k}$ and $\tilde{X_{kl}}$ cointain $E^c$ their intersection is non trivial and hence $V=X_k\cap X_{kl}$ is as well. Since $X_k$ is $A^k$ invariant, $V$ is an $A^{kl}$ invariant non-trivial subspace of $X_{kl}$. Hence $V=X_{kl}$, implying $X_{kl}\subset X_k$, but since $d_k$ was minimal $X_{kl}=X_k$.\newline
    Since $X=X_{kl}$ we have that any non-zero integer vector of $X$ is $A^{kl}$-cyclic as well, hence $A^k$ is Pseudo-Anosov on $X$.\newline
\end{proof}

\begin{lemma}
    Let $p(x)$ be an irreducible monic polynomial in $\mathbb{Z}[x]$ with constant term 1. Suppose that $p(x)$ has a unitary root that is not a root of unity. Then, $\deg(p(x))$ is even and, in particular, $p(x)$ is a reciprocal polynomial. Furthermore, $p(x)$ has at least a root of modulus less then 1 and one of modulus greater than 1.
\end{lemma}

\begin{proof}
    This is essentially lemma A.3 in \cite{Hertz2005}, we present an alternative proof.\newline
    Let $\xi$ be the non-real unitary complex root of $p(x)$. Then $1/\xi=\bar{\xi}$ is a root of $p(x)$. Setting $n=\deg(p(x))$, we have that $x^np(1/x)$ is a monic irreducible polynomial in $\mathbb{Z}[x]$ that vanishes in $1/\xi$. By uniqueness of the carachteristic polynomial we must have that $p(x)=x^np(1/x)$, meaning that it is reciprocal. Therefore, the roots of $p(x)$, all distinct beacuse it is irreducible, comes in pairs of inverses, since $\pm1$ are not roots of $p(x)$. This implies that $\deg(p(x))$ is even.\newline
    Finally if all roots of a monic integer coefficients polynomial are unitary they must be roots of unity (Kroenecker's theorem, see \cite{Greiter1978}).\newline
\end{proof}

\begin{lemma}
    We have that $\text{dim(X)}$ must be even and $\dim(X)\geq 4$.
\end{lemma}

\begin{proof}
    By construction of $X$, the carachteristic polynomial $p_{A^k|_X}(x)$ of $A^k|_X$ is irreducible over $\mathbb{Z}$. Since $\dim(X)=\deg(p_{A^k|_X}(x))$, it must be even because of the previous lemma, applied to $p_{A^k|_X}(x)$. The same lemma rules out dimension 2, because $p_{A^k|_X}(x)$ has two unitary roots.
\end{proof}

\begin{lemma}\label{dimension 7 polynomial}
    Let $A\in SL(7,\mathbb{Z})$ have no roots of unity as eigenvalues.\newline
    Then, either $\dim(E^c)=0$ or $\dim(E^c)=2$.
\end{lemma}

\begin{proof}
    We have that $\dim(E^c)$ equals the number of unitary roots of the carachteristic polynomial $p_A(x)$ of $x$. Let $q(x)$ be an irreducible factor of $p_A(x)$ in $\mathbb{Z}[x]$ that has a unitary root. Therefore, $\deg(q(x))$ is even and at least 4. This proves that $p_A(x)$ can have at most one such factor, hence $q(x)$ has all the unitary roots of $p_A(x)$. If $\deg(q(x))=6$, then $p_A(x)$ would have a rational root, which is impossible. The degree of $q(x)$ must be 4, and consequentely it has at most 2 unitary roots.
    \newline
\end{proof}

\subsection{Diophantine Estimates}

Up to a bi-Lipschitz change of coordinates we can assume that $(X,\Gamma)$ from the previous lemma is $(\mathbb{R}^d,\mathbb{Z}^d)$ for a suitable $d$. Therefore, many results about Psuedo-Anosov linear maps of $\mathbb{R}^N$ from \cite{Hertz2005} can be directly transposed to our setting. Here we also exploit the fact that the following statements are invariant if one swaps $A$ with $A^k$, where $k$ comes from Lemma \ref{Pseudo Anosov subspace}.\newline

\begin{lemma}\label{Diphantine center}
    There exists a constant $c'=c'(A)>0$, such that, calling $r=\frac{\text{dim}X}2$,
    \begin{equation*}
        |n^c| \geq\frac {c'}{|n|^{r}} \quad\quad \text{for any }n\in\Lambda, n\neq0.
    \end{equation*}
\end{lemma}

\begin{proof}
    This is lemma A.10 in \cite{Hertz2005} with $\delta=1/2$.\newline
\end{proof}

We now fix a $\mathbb{Z}$-basis of $\Lambda$, namely $\{e_1,\dots, e_{\dim(X)}\}$.\newline
We also consider the linear trasformation $R:E^c\rightarrow \mathbb{R}^2$ defined by $R(e_1^c)=(1,0)$ and $R(e_2^c)=(0,1)$. The next two lemmas are Lemma A.11 and Lemma A.12 in \cite{Hertz2005}.

\begin{lemma}\label{Diphantine condition 1}
    If $\dim(X)\geq6$ there exists $n\in\Lambda$ such that if we call $R(n^c)=\alpha$, then for any $\delta>0$, there exists a constant $c'>0$, such that
    \begin{equation*}
        ||| k\alpha||| \geq \frac c{k^{2+\delta}}, \quad\quad \forall k\in\mathbb{Z},k\neq0.
    \end{equation*}
\end{lemma}

Unfortunately this lemma is not true in dimension 4, hence, in that case, we will have to rely on the following alternative one.

\begin{lemma}\label{Diphantine condition 2}
    If $\dim(X)=4$ there exist $n_1,n_2\in\Lambda$ such that if we call $R(n_1^c)=\alpha_1$ and $R(n_2^c)=\alpha_2$, then there exists a constant $c'>0$, such that
    \begin{equation*}
        \max_{i=1,2}|||k\cdot\alpha_i||| \geq \frac c{k^2}, \quad\quad \forall k\in\mathbb{Z}^2,k\neq0.
    \end{equation*}
\end{lemma}

\section{Proof of the Minimality Criterion}\label{Proof of the Minimality Criterion}

We focus on a connected component $U$ of $V$, with the goal of showing that $U=\mathbb{R}^N$.

\subsection{The structure of U}\label{the structure of U}

Our first goal is to prove that $U$ is invariant by a discrete subgroup of directions that spans $E^c$ and that $U$ is simply connected. Here we follow ideas from \cite{Hertz2005}.\newline

\noi Given $\varepsilon>0$ we set $L(\varepsilon)=\varepsilon^{-2}$, then define $W_{\varepsilon}(x)=W^s_{\varepsilon}(W^u_{L(\varepsilon)+\varepsilon}(W^s_{L(\varepsilon)}(W^c_{\varepsilon}(x))))$.

\begin{lemma}\label{infinite volume}
    Provided that $f$ is sufficiently $\mathcal{C}^1$-close to $A$ the following holds.\newline
    For every $x\in\mathbb{R}^N$, we have that $\text{Vol}(W_{\varepsilon}(x)\cap(X+x))\rightarrow +\infty$ when $\varepsilon\rightarrow0$.
\end{lemma}

\begin{proof}
    The space $X/E^c$ inherits naturally a distance from $X$, and so does $X/E^c+x$. We set $2R=L(\varepsilon)$ and consider the ball $B=B(x,R)$ in $X/E^c+x$. Setting $2\delta=\varepsilon L(\varepsilon)^{-2\beta}$, we can find a $2\delta$-separated subset $\{x_1,\dots,x_N\}$ of $B$, where $N\geq C (R/\delta)^{\text{dim}(X)-2}$, where $C$ depends only on $\text{dim}(X)$.\newline
    By lemma \ref{proper 2}, there exist points $y_1,\dots,y_N$, such that, for every $i$, $y_i\in E^c(x_i)$ and there exists a 2-pieces $su$-curve $\gamma_i$, that joins $x_i$ and $y_i$ and satisfies $\ell(\gamma_i)<L(\varepsilon)$.\newline
    We notice that $\{y_1,\dots, y_N\}$ is a $2\delta$-separated set in $X$.\newline
    In addition, we have that $y_i \in W^u_{L(\varepsilon)}(W^s_{L(\varepsilon)}(x))$. Lemma \ref{Holonomies Lip} implies that $W^c_{\delta}(y_i)\subset W^u_{L(\varepsilon)}(W^s_{L(\varepsilon)}(W^c_{\varepsilon}(x)))$, therefore $W^u_{\delta}(W^s_{\delta}(W^c_{\delta}(y_i)))\subset W_{\varepsilon}(x)$. Now, by lemma \ref{proper 1}, we get that $B(y_i,\delta)\subset W_{\varepsilon}(x)$. Since the balls $B(y_i,\delta)$ are disjoint we get that
    \begin{align*}
        \text{Vol}(W_{\varepsilon}(x)\cap(X+x)) &\geq \sum_{i=1}^N \text{Vol}(B_{X+x}(y_i,\delta)\\
        & = C \delta^{\text{dim}(X)}\frac{L(\varepsilon)^{\text{dim}(X)-2}}{\delta^{\text{dim}(X)-2}}\\
        &= C\varepsilon^{2-2(\text{dim}(X)-2-4\beta)}.
    \end{align*}
    If $\beta<1/4$ (i.e. $f$ is sufficinetly close to $A$), the exponent is negative and this concludes the proof.

\end{proof}

\begin{lemma}\label{n epsilon}
    For every $\varepsilon>0$ small enough and $x\in\mathbb{R}^N$, there exists $n=n_{\varepsilon}\in \Lambda\subset X$ such that $W_{\varepsilon}(x)\cap(W_{\varepsilon}(x)+n_{\varepsilon})\neq\emptyset$. Moreover, if $B(x,\varepsilon)\subset U$, then  $U+n_{\varepsilon}=U$.\newline
    Finally, $|n_{\varepsilon}|\leq 5(1+\kappa) L(\varepsilon)$
\end{lemma}

\begin{proof}
    Consider $\varepsilon$ small enough to ensure that $\text{Vol}(W_{\varepsilon}(x)\cap(X+x))>\text{Vol}(X/\Lambda)$, which is finite since $\Lambda$ has full rank. We now consider the projection $\pi:(X+x)\rightarrow X/\Lambda$. If the fist part of the statement did not hold,we would then have that $\pi|_{W_{\varepsilon}(x)\cap(X+x)}$ is injective, but this would contraddict the fact that $\pi$ is a local isometry and that $\text{Vol}(W_{\varepsilon}(x)\cap(X+x))>\text{Vol}(X/\Lambda)$.\newline
    Now we notice that $U+n_{\varepsilon}$ is a connected component of $V$ (since $V$ is $\mathbb{Z}^n$ invariant) and that the first part ensures that $U+n_{\varepsilon}\cap U\neq \emptyset$, giving that $U+n_{\varepsilon}=U$.\newline
    Finally we must have $|n_{\varepsilon}|\leq \text{diam}(W_{\varepsilon}(x))$, hence the last statement.
\end{proof}

\begin{lemma}
    $U+\Gamma = U$ for a suitable $\Gamma\subset\Lambda$, full rank subgroup of $X$.
\end{lemma}

\begin{proof}
    (Same as in [RH], expanded). The non-wandering set of $f$ is the whole $\mathbb{T}^N$ and $U$ is open. This implies the existence of $l\in\mathbb{N}$ and $h\in\mathbb{Z}^N$ such that $(F^l(U)+h)\cap U\neq\emptyset$. Since $V$ is $F$-invariant and $\mathbb{Z}^N$-invariant, $F^l(U)+h$ is a connected component of $V$, and therefore $U$, since their intersection is non-empty. To recap $U= F^l(U)+h$.\newline
    Using the fact that $F=A+G$, and that $A$ is linear and $G$ is $\mathbb{Z}^N$-periodic we can show by induction that $F^l(U+n)=F^l(U)+A^ln$.\newline
    Combining the two and lemma \ref{n epsilon}, we can write
    \begin{align*}
        U &= F^l(U) + h\\
        &= F^l(U+n) + h\\
        &= F^l(U) + A^ln + h\\
        &= U + A^ln.
    \end{align*}
By induction we get that $U$ is invariant by $A^{sl}n$ for every $s\in\mathbb{N}$.\newline
We set $\Gamma=\langle n, A^{kl}n, \dots, A^{kl(\dim(X)-1)}n\rangle$, which, by lemma \ref{Pseudo Anosov subspace}, is a full rank subset of $\Lambda$, since $n$ is cyclic for $A^{kl}$ and, by our proof, $U+\Gamma=U$. ($k$ is from Lemma \ref{Pseudo Anosov subspace}.)
\end{proof}

\begin{remark}
    Here we used Poincar\'e recurrence (and hence the volume preserving assumption) to see that the non-wandering set is full.
\end{remark}

\begin{lemma}
    $U$ is simply connected.
\end{lemma}

\begin{proof}
    A detailed proof is given in appendix \ref{Simply connected proof}.\newline
    The proof will mostly follow the one in \cite{Hertz2005}, with the only changes occurring in the proof of their lemma 4.6. Their proof exploits an object very similar, yet different, to $W_{\varepsilon}(x)$. The modifications we introduce in the proof make up for the slight differences between these two objects.

\end{proof}

\subsection{Conclusion of the proof}\label{main proof}

From the properties proved in the previous subsection, we conclude the proof with a topological argument. We actually prove the following statement, from which Theorem \ref{minimality} follows directly.

\begin{proof}
\begin{lemma}
    If $f$ is sufficiently close to $A$ in $\mathcal{C}_{\text{vol}}^1(\mathbb{T}^N)$, for every $V\subset\mathbb{R}^n$ $su$-saturated, $F$-invariant $\mathbb{Z}^N$-invariant non-empty open set, $V=\mathbb{R}^N$.
\end{lemma}

In view of lemma \ref{proper 2} and the $su$-saturation of $U$, to conclude the proof it is enough to show that $U$ is $E^c$-saturated.\newline

Fix $x\in U$ and $y\in E^c(x)$. We want to conclude that $y\in U$.\newline
The main idea is to combine the connectedness of $U$ and its $\Gamma$-invariance to build a curve in $U$ that wraps around $W^s(W^u(y))$. Then exploit the simple connectedness of $U$ to deduce that it must intersect $W^s(W^u(y))$, which yields the conclusion, by its $su$-saturation.\newline

\noindent We first need a preliminary lemma.\newline
Let us define $C_{\varepsilon}=\{z\in \mathbb{R}^n: |z^{su}|<\varepsilon|z|\}$ and $C_{y,\varepsilon}=y+C_{\varepsilon}$.

\begin{lemma}\label{piecewise linear curve}
    Given $x\in U$ and $y\in E^c(x)$ For every $\varepsilon>0$, there exist, $n_1,n_2,n_3\in \Gamma$ such that, for every $R>0$ there exist $M\in\mathbb{N}$ and a closed piecewise linear curve $\gamma:[0,M]\rightarrow \mathbb{R}^n\setminus E^{su}(y)$, such that:
    \begin{enumerate}
        \item $\gamma(i)\in x+\Gamma$, for $i=0,\dots,M$;
        \item $\gamma'|_{(i,i+1)}\equiv n_{j(i)}$ for a suitable $j(i)\in\{1,2,3\}$, for $i=0,\dots,M$;
        \item $\pi_1(\mathbb{R}^n\setminus E^{su}(y))=(\gamma)$;
        \item $\gamma(t)\in C_{y,\varepsilon}$ and $|\gamma(t)-y|>R$ for every $t$.
    \end{enumerate}
\end{lemma}

\begin{proof}
    We can suppose that $x=0$.\newline
    Define $d_{\Gamma}=\text{diam}(X/\Gamma)$. Take three points $z_1,z_2,z_3\in E^c$, such that $B(0,3d_{\Gamma}/\varepsilon)$ lies inside their convex hull. We can find $v_1,v_2,v_3\in\Gamma$ such that $|v_i-z_i|<d_{\Gamma}$. Therefore, $B(0,2d_{\Gamma}/\varepsilon)$ lies in the convex hull of $v_1^c, v_2^c,v_3^c$. This and the fact that $|v_i^{su}|<d_{\Gamma}$ imply that $\overline{v_iv_{i+1}}\subset C_{\varepsilon/2}$ (indices must be intended modulo 3).\newline
    Set $n_i=v_{i+1}-v_i$ (this does not depend on $R$).\newline
    We now take $K\in\mathbb{N}$, whose size will be specified later.\newline
    Consider $w_i=Kv_i$, and define the closed curve $\gamma:[0,3K]\rightarrow\mathbb{R}^n$, such that $\gamma(i)=w_i$, and $\gamma$ linear between these points.\newline
    Conditions 1 and 2 hold for any $K$.\newline
    For $K>(2d_{\Gamma}/\varepsilon)/|x-y|$, we have that $y$ lies in the convex hull of $w_1^c,w_2^c,w_3^c$, hence condition 3.\newline
    We notice that $\gamma(t)\in C_{\varepsilon/2}$ and $|\gamma(t)|>K(2d_{\Gamma}/\varepsilon)$ for every $t\in[0,3K]$.\newline
    For $K>|x-y|/d_{\Gamma}$, we have that $\gamma(t)\in C_{\varepsilon/2}\cap B(0,2|x-y|/\varepsilon)^c\subset C_{y,\varepsilon}$.\newline
    Finally, we pick $K$ large enough to ensure that $|\gamma(t)|>K(2d_{\Gamma}/\varepsilon)>R+|y|$, which is what we need to meet condition 4.\newline
    This concludes the proof of the lemma.\newline
\end{proof}

For the rest of the proof, we can suppose that $x\in E^c$ and $y=0$, for the sake of notation. Since $U$ is $su$-saturated, we assume, by contradiction, that $U\cap W^u(W^s(0)))=\emptyset$.\newline
We apply Lemma \ref{piecewise linear curve}.
Being $U$ connected there exist a path $\tilde\gamma_j:[0,1]\rightarrow U$ that joins $x$ and $x+n_j$ for $j=1,2,3$. By compactness, there exists a constant $C_{n_1,n_2,n_3}>0$ (that depends only on $n_1,n_2,n_3$) such that $|(\tilde\gamma_j(t)-x)-t\gamma_j)|<C$ for every $t\in[0,1]$. We set $R=2C_{n_1,n_2,n_3}/\varepsilon$. We can now define a new curve $\tilde\gamma:[0,M]\rightarrow \mathbb{R}^n$ by replacing its linear pieces with $\Gamma$-translates of the curves $\tilde\gamma_j$. More precisely
    \begin{equation}
        \gamma(t) = \gamma(i) + \tilde\gamma_{j(i)}(t)-x   \qquad \text{for }t\in[i,i+1).
    \end{equation}
Since $U$ is $\Gamma$-invariant, $\tilde\gamma:[0,M]\rightarrow U$. Moreover $\|\gamma-\tilde\gamma\|_{\infty}<C_{n_1,n_2,n_3}$. This implies that $\tilde\gamma(t)\in C_{2\varepsilon}$ for every $t$, and at the same time that, $\gamma$ and $\tilde\gamma$ are homotopic in $\mathbb{R}^n\setminus E^{su}$, implying that $\pi_1(\mathbb{R}^n\setminus E^{su})=(\tilde\gamma)$.\newline
Recall the map $\Phi_y$ from section \ref{Su-saturation}, and set $\Phi=\Phi_0$. By lemma \ref{homomorphism}, we can apply $\Phi^{-1}.$ Now, $\Phi^{-1}(W^u(W^s(0)))=E^{su}$ and we define $\hat\gamma=\Phi^{-1}\circ\tilde\gamma$. If $\kappa$ and $\varepsilon$ are small enough we have that $|\hat\gamma(t)-\tilde\gamma(t)|<|\tilde\gamma(t)^c|$ for every $t$. This implies that $\hat\gamma$ is homotopic to $\tilde\gamma$ in $\mathbb{R}^n\setminus E^{us}$ and therefore its homotopy type is nontrivial in $\mathbb{R}^n\setminus E^{us}$. Finally, we have that $\Phi^{-1}(U)\cap E^{su}=\emptyset$ and $\text{Im}(\hat\gamma)\subset \Phi^{-1}(U)$. Since $U$ has trivial fundamental group and $\Phi$ is an homomorphism, $\hat\gamma$ should be contractible in $\Phi^{-1}(U)$, but then, a fortiori, it should be contractible in $\mathbb{R}^n\setminus E^{su}$, providing the desired contradiction.
\end{proof}

\section{Proof of Stable Ergodicity}\label{proof of ergodicity}

This section is devoted to the proof of theorem \ref{stable ergodicity}.\newline
As mentioned, the proof follows ideas from \cite{Hertz2005}, with our theorem \ref{minimality} playing the role of theorem 4.1 in \cite{Hertz2005}. Our exposition takes advantage of more recent results about accessibility classes.\newline

Our goal is to prove the ergodicity of $f$ by showing that it is either accessible or essentially accessible. Hence we start by describing the structure of the accessibility classes.\newline

For $x\in\mathbb{T}^N$ we denote its accessibility class with respect to $f$ by $AC_{f}(x)$, while for $\tilde x\in \mathbb{R}^N$ we denote the one with respect to $F$ by $AC_F(\tilde x)$. When no confusion arises, we will drop this indices. If $\tilde x$ is a lift of $x$ we have that $\pi^{-1}(AC_{f}( x)) = AC_F(\tilde x)+\mathbb{Z}^N$ and, in particular, $AC_F(\tilde x)$ is the arc-connected component of $\pi^{-1}(AC_{f}( x))$ containing $\tilde x$.\newline

We remark that if $f$ is close enough to $A$, $f$ is dynamically coherent and center bunching. Therefore, we can apply theorem B from \cite{rodriguez2017structure}.

\begin{theorem}
Let $f : M \rightarrow M$ be a partially hyperbolic $\mathcal{C}^2$ diﬀeomorphism on a closed Riemmaninan manifold. Suppose $f$ to have a two-dimensional center bundle, and to be dynamically coherent and center bunching.\newline
Then all accessibility classes are injectively immersed $\mathcal{C}^1$-submanifolds.
\end{theorem}

We can then consider the quantity $\text{codim}(AC(x))$. By our previous consideretions we also have that $\dim(AC_F(\tilde x))=\dim AC_f(x)$.\newline We investigate how it varies in $x$, with the goal of proving that it must be constant.\newline
The first remark is that $\text{codim}(AC(x))=\text{codim}_{W^c(0)}(AC(x)\cap W^c(0))$.

\begin{lemma}\label{AC invariant homotopy}
    Working in $\mathbb{R}^N$, consider $x\in W^c(0)$ and $y\in W^c(0)\cap AC(x)$.\newline
    There exists a continuous homotopy $H:W^c(0)\times[0,1]\rightarrow W^c(0)$, such that $H(x,0) = x$, $H(x,1) = y$, $H(z,[0,1]) \subset AC(z)$ for any $z\in W^c(0)$. 
\end{lemma}

\begin{proof}
    Let $\gamma:[0,1]\rightarrow \mathbb{R}^N$ be a $su$-curve such that $\gamma(0)=x$ and $\gamma(1)=y$. We define $H(t,z)=\pi^{su}(h^{\gamma_{|[0,t]}}(z))$. First notice that $\Im(H)\subset\Im(\pi^{su})=W^c(0)$. Then, we have that $H(x,0)=\pi^{su}(x)=x$ and that $H(1,x)=\pi^{su}(h^{\gamma}(x))=\pi^{su}(y)=y$. Finally, the last requirement is met since both $h^{\gamma|_{[0,t]}}$ and $\pi^{su}$ preserve the accessibility classes.\newline
\end{proof}

This implies that $AC_F(x)\cap W^c(0)$ is arc-connected for every $x\in W^c(0)$.\newline
In particular, if $\dim(AC(x)\cap W^c(0))=0$, we must have that $\#AC(x)\cap W^c(0)=1$. 

\begin{lemma}
    The map $x \mapsto \text{codim}(AC(x))$ is upper semi-continuous on $\mathbb{T}^N$ and $\mathbb{R}^N$.
\end{lemma}

\begin{proof}
    The set $\{x\in\mathbb{R}^N:\text{codim}(AC(x))=0\}$ is open by definition.\newline
    We are left with proving that $\{x\in\mathbb{R}^N:\text{codim}(AC(x))\leq 1\}$ is open as well. For $x\in W^c(0)$ belonging to such set, we can find $y\neq x$ in $AC(x)\cap W^c(0)$. We can then apply lemma \ref{AC invariant homotopy}, and obtain the described homotopy $H$. Since $H(x,0)\neq H(x,1)$, we have that there exists $\varepsilon>0$ such that $H(z,0)\neq H(z,1)$ for every $z\in W^c_{\varepsilon}(x)$. Therefore, for every such $z$, $\text{codim}(AC(z))\leq 1$. This ends the proof.\newline
    The statement for $\mathbb{T}^N$ follows since $\pi:\mathbb{R}^N\rightarrow \mathbb{T}^N$ is open.
\end{proof}

\begin{corollary}
    All accessibility classes have the same dimension.
\end{corollary}

\begin{proof}
    The sublevels $\{x\in\mathbb{R}^N:\text{codim}(AC(x))\leq k\}$ are open, $su$-saturated, $F$-invariant and $\mathbb{Z}^N$-invariant. By theorem \ref{minimality}, each of these sets must be either empty or full, hence the conclusion.
\end{proof}

\begin{lemma}
    If $f$ is sufficiently $\mathcal{C}^1$-close to $A$, then $\text{codim}(AC(0))\neq 1$.
\end{lemma}

\begin{proof}
    Since $E^c$ contains no real eigenvector of $A$, if $f$ is sufficiently close to $A$, $D_0F$ is close enough to $A$ (in the space of normed linear operators) to guarantee that there are no real eigenvectors of $D_0F$ in a suitable open cone $U$ around $E^c$. Again, if $f$ is sufficiently close to $A$, $T_0W^c(0)$ is so close to $E^c$ to be contained in $U$. Therefore, no line of $T_0W^c(0)$ can be invariant under $D_0F$ contradicting the fact that $W^c(0)\cap AC(0)$ is $F$-invariant and has dimension 1.\newline
\end{proof}

We can now conclude the proof of Theorem \ref{stable ergodicity}.

\begin{itemize}
    \item If $\text{codim}(AC(x))=0$ for every $x$, all accessibility classes are open and by connectedness there is only one of them. Being $f$ accessible, it is also ergodic by Burns Wilkinson theorem.

    \item If $\text{codim}(AC(x))=2$ for every $x$, one can work exactly as in section $6$ of \cite{Hertz2005}.\newline
    We just describe the general strategy. The joint integrability $\mathcal{F}^s$ and $\mathcal{F}^u$ imply that $T_n\circ T_m=T_{n+m}$, i.e. $(T_n)_n$ induces a $\mathbb{Z}^N$ action on $E^c$.  Lemma \ref{Tn close to linear} says that the maps $T_n$ are close to traslations in $\mathcal{C}^{22}$ and lemmas \ref{Diphantine condition 1} and \ref{Diphantine condition 2} guarantee that for suitable choices of $n$ these translations are Diophantine. By KAM theory it is possible to linearize this whole $\mathbb{Z}^N$ action. This implies a $\mathcal{C}^1$ conjugacy between the stable and unstable foliations of $f$ and $A$. This implies that the essential accessibilty of A is inherited by $f$, which, in turn, implies the ergodicity.
\end{itemize}

\newpage

\appendix

\section{$U$ is simply connected}\label{Simply connected proof}


Recall the map $\pi^{su}:\mathbb{R}^N\rightarrow W^c(0)$, from subsection \ref{Invariant Manifolds and Holonomies}. It is a fibration (lemma 4.3 in \cite{Hertz2005}), and, as a consequence, we have the following (lemma 4.4 in \cite{Hertz2005}).

\begin{lemma}
    Given any open and connected su-saturated set $E$,
    \begin{equation*}
        \pi_1 (E) = \pi_1 (E \cap W^c(0)).
    \end{equation*}
\end{lemma}

The proof then reduces to showing that $U\cap W^c(0)$ is simply connected. We want to exploit the following algebraic topolgy lemma.

\begin{lemma}
    Let $E\subset\mathbb{R}^2$ be a connected open set such that all connected components of $\mathbb{R}^2\setminus E$ are unbounded. Then $E$ is simply connected.
\end{lemma}



\begin{proof}
    Suppose that $E$ is not simply connected. Since $E$ is an open subset of $\mathbb{R}^2$ we can find  Jordan curve $\gamma:[0,1]\rightarrow E$, which is non-contractible in $E$.
    By Jordan curve theorem there exists a disk $D\subset\mathbb{R}^2$ such that $\Im(\gamma)=\partial D$. If $D\subset E$, then $\gamma$ would be contractible in $E$. Hence a point of $D$ must belong to $\mathbb{R}^2\setminus E$, but then its connected component must be contained in $D$ and therefore bounded. This is the desired contradiction. \newline
\end{proof}

Let us consider $C=E^c\setminus (\sigma_0^c)^{-1}(U)$. Recalling that $\sigma_0^c$ is a bi-Lipschitz homeomorphism and identifying $E^c$ with $\mathbb{R}^2$, in view of the previous lemma, to conclude the proof we only need to prove that all connected components of $C$ are unbounded.

\begin{lemma}
    For any $ x\in E^c$ and $\delta>0$ there exist $n\in\Lambda$ with $U+n=U$, $k\in\mathbb{N}$, and curves $\eta_i:[0,1]\rightarrow E^c$, for $i=1,\dots,k$, such that, calling $\hat\eta_i=\sigma_0^c\circ\eta_i$ and $\hat x=\sigma_0^c( x)$:
    \begin{enumerate}
        \item $\hat\eta_i([0,1]) \subset (AC(\hat x)+\Lambda)\cap W^c(0)$;
        \item $|\eta_i(0)-T_{(i-1)n}( x)|<\delta$ and $\hat\eta_i(1) = T_{in}(\hat x)$;
        \item $|T_{kn}(x)- x|\rightarrow +\infty$, when $\delta\rightarrow 0$.
    \end{enumerate}
\end{lemma}

\begin{proof}
    We recall the constants $\beta$ from lemma \ref{Holonomies Lip} and $r$ from lemma \ref{Diphantine center}. We set $s=2r+1$ and $\gamma=1-\beta(s+14)$.\newline
    If $f$ is sufficiently $\mathcal{C}^1$-close to $A$, $\beta$ approaches $0$ and $\gamma$ is positive.\newline
    Given $\delta>0$, we consider $\varepsilon>0$ such that $\varepsilon^{\gamma}<\delta$ (notice that $\varepsilon\rightarrow0$, whenever $\delta\rightarrow0$). We obtain $n=n_{\varepsilon}$ from lemma \ref{n epsilon} (with $|n_\varepsilon|<5(1+\kappa)L(\varepsilon)$).\newline
    Finally we pick $k\in\mathbb{N}$ such that $\varepsilon^{-s}/2<k<\varepsilon^{-s}$.\newline

    We start by proving item 3.\newline
    Exploiting lemma \ref{Tn log} at the first line, and lemma \ref{Diphantine center} at the second one we have the following. (In the next computations we denote all universal constants with $C$, which might assume a diffent value, even within the same line.)
    \begin{align*}
        |T_{kn}( x) - (x + n^c)| &\geq (|kn^c|) - C\log|kn| - C\\
        &\geq k\frac{C}{|n|^r} - C\log(k)  - C\log(|n|) - C\\
        &\geq C \frac{\varepsilon^{-s}}{\varepsilon^{2r}} - sC\log(\varepsilon^{-1}) - 2C\log(\varepsilon^{-1})-2C\log(5C) -C\\
        &\geq C \varepsilon^{-1} - (2C + sC)\log(\varepsilon^{-1}) - C\\
        &\rightarrow+\infty.
    \end{align*}

    In order to construct the curves $\eta_i$, we first construct their starting points $\eta_i(0)= y_i$. We aim to find points $ y_i\in E^c$ such that, again denoting $\hat y_i=\sigma_0^c( y_i)$:
    \begin{itemize}
        \item[a)] $|T_{(i-1)n}( x)- y_i|<\delta$;
        \item [b)] $ \hat y_i\in AC ( \hat x+in)$.\newline
    \end{itemize}

    Recalling that, $W_{\varepsilon}(\hat x+n) = W_{\varepsilon}(\hat x)+n$, by lemma \ref{n epsilon}, we have that $W_{\varepsilon}(\hat x +n ) \cap W_{\varepsilon}(\hat x)\neq\emptyset$. This implies the existence of a 5-legged $su$-path $\gamma:[0,1]\rightarrow \mathbb{R}^N$ such that $\gamma(0)\in W_{\varepsilon}^c(\hat x + n)\subset W^c(n) $ and $\gamma(1)\in W_{\varepsilon}^c(\hat x)\subset W^c(0)$.\newline
    We call $\hat S=h^{\gamma}:W^c(n)\rightarrow W^c(0)$ the $su$-holonomy along $\gamma$. Notice that $\gamma$ is shorter that $10L(\varepsilon)$ and consequently, by lemma \ref{Holonomies Lip}, $\text{Lip}(\hat S)\leq C(L(\varepsilon))^{5\beta}$.\newline
    We call $\hat T_{m} = \sigma_0^c\circ T_m \circ (\sigma_0^c)^{-1}: W^c(0)\rightarrow W^c(0)$. By lemma \ref{Holonomies Lip}, $\text{Lip}(\hat T_m)\leq C|m|^{\beta}$.\newline

    We define $\hat y_i= \hat T_{(i-1)n}\circ \hat S(\hat x +n)$.\newline

    By definition of $\hat S$, we have that $\hat S (\hat x +n)\in AC(\hat x +n)$. Moreover, for any $\hat z \in W^c(0)$ and any $m\in\mathbb{Z}^N$, we have that $\hat T_m(\hat z) \in AC(\hat z + m)$. We conclude that $\hat y_i= \hat T_{(i-1)n}\circ \hat S(\hat x +n)\in AC(\hat x + (i-1)n +n)=AC(\hat x +in)$, satisfying the second condition we required on $ y_i$.\newline

    We now focus on the first condition on $y_i$.\newline
    Indeed, for $\varepsilon$ small enough, we have the following computation.

    \begin{align*}
        |y_i-T_{(i-1)n}(x)| & \leq |\hat y_i-\hat T_{(i-1)n}(\hat x)|\\
        & \leq \text{Lip}(\hat T_{(i-1)n}) | \hat S (\hat x +n) - \hat x|\\
        & \leq \text{Lip}(\hat T_{(i-1)n}) (| \hat S (\hat x +n) - \hat S(\gamma(0))| + | \gamma(1) - \hat x|)\\
        & \leq \text{Lip}(\hat T_{(i-1)n}) (\text{Lip}(\hat S)\varepsilon + C\varepsilon)\\
        & \leq C |kn|^{\beta} L(\varepsilon)^{5\beta} \varepsilon\\
        & \leq C \varepsilon^{-s\beta}\varepsilon^{-2\beta}\varepsilon^{-10\beta}\varepsilon\\
        &\leq \varepsilon^{1-\beta(14+s)} = \varepsilon^{\gamma} < \delta.
    \end{align*}

Now we can build the curves $\hat \eta_i$.\newline
Since $\hat y_i\in AC(\hat x + in)$, there exists a $su$-curve $\tilde \eta_i$, such that $\tilde\eta_i(0)=y_i$ and $\tilde\eta_i(1)=\hat x +in$. Of course, $\tilde \eta([0,1)]\subset AC(\hat x+in)$.\newline
We set $\hat \eta_i = \pi^{su} \circ \tilde \eta_i$.\newline
Since we have that $\pi^{su}(\mathbb{R}^N)=W^c(0)$ and $\pi^{su}$ preserves the accessibility classes, $\hat\eta_i$ satisfies item 1.\newline 
Moreover, we have that $\hat\eta_i(1)=\pi^{su}(\tilde\eta_i(1))=\pi^{su}(\hat x +in)=\hat T_{in}(\hat x)$, hence $\eta_i(1)=T_{in}(x)$. Finally, since $\hat y_i\in W^c(0)$, we have that $\hat\eta_i(0)= \pi^{su}(\tilde\eta_i(0))=\pi^{su}(\hat y_i)=\hat y_i$, and, therefore, $\eta_i(1)=y_i$ and $\eta_i$ satisfies item 2.\newline
\end{proof}

We can now conclude the proof following \cite{Hertz2005} (see the proof of Corollary 4.7).\newline
Let $x$ be a point in $C$. Suppose, by contradiction that its component of $C$ is contained in $B(x,R)$ for some $R>0$. For $\delta>0$ we apply the previous lemma to $x$.\newline
Since $U$ is $n$-invariant, $T_{in}(x)\in C$ for every $i=0,\dots,k$. In addidion, the $su$-saturation of $U$ implies that $\eta_i([0,1])\subset C$ for $i=0,\dots,k$.\newline
We define
\begin{equation*}
    \tilde K_\delta = \bigcup_{i=0}^k \overline{B(T_{in}(x),\delta)} \cup \bigcup_{i=1}^k \eta_i([0,1]),
\end{equation*}
and $K_{\delta}$ to be the connected component of $\tilde K_{\delta}\cap \overline{B(x,R)}$ that contains $x$.\newline
Since $\tilde K_{\delta}$ is connected and $|T_{kn}(x)-x|>R$ for $\delta$ small enough, we have that
\begin{equation*}
    K_{\delta}\cap \partial B(x,R)\neq \emptyset \quad \text{and that} \quad K_{\delta} \subset C+\overline{B(0,\delta)}.
\end{equation*}
By the compactness of the space of the compact subsets of $\overline{B(x,R)}$ with respect to the Hausdorff distance, we have the existence of a sequence $\delta_m\rightarrow 0$ and a compact subset $K$ of $\overline{B(x,R)}$ such that $K_{\delta_n}\rightarrow K$ in the Hausdorff metric.\newline
Passing to the limit, $K$ must be connected, contain $x$ and intersect $\partial B(x,R)$. At the same time we have that $K\subset C +\overline{B(0,\delta)}$ for every $\delta>0$, implying that $K\subset C$, in contradiction with our initial assumption.

\newpage

\printbibliography

@article{Hertz2005,
title = "Stable ergodicity of certain linear automorphisms of the torus",
author = "Hertz, Federico Rodriguez",
year = "2005",
month = jul,
doi = "10.4007/annals.2005.162.65",
language = "English (US)",
volume = "162",
pages = "65--107",
journal = "Annals of Mathematics",
issn = "0003-486X",
publisher = "Department of Mathematics at Princeton University",
number = "1",
}

@article{Greiter1978,
    author = "Greiter, G.",
    title = "A Simple Proof for a Theorem of Kronecker",
    journal = "The American Mathematical Monthly",
    volume = "85(9)",
    pages = "756–-757",
    year = "1978",
    doi = "10.1080/00029890.1978.11994694",
}

@article{rodriguez2017structure,
  title={Structure of accessibility classes},
  author={Rodriguez-Hertz, Jana and V{\'a}squez, Carlos H},
  journal={arXiv preprint arXiv:1706.01156},
  year={2017}
}

@article{burns2010ergodicity,
  title={On the ergodicity of partially hyperbolic systems},
  author={Burns, Keith and Wilkinson, Amie},
  journal={Annals of Mathematics},
  pages={451--489},
  year={2010},
  publisher={JSTOR}
}

@article{halmos1943automorphisms,
  title={On automorphisms of compact groups},
  author={Halmos, Paul R},
  year={1943}
}

@article{brin1974partially,
  title={Partially hyperbolic dynamical systems},
  author={Brin, Michael I and Pesin, Ja B},
  journal={Mathematics of the USSR-Izvestiya},
  volume={8},
  number={1},
  pages={177--218},
  year={1974}
}

@book{hirsch1977invariant,
    author={Hirsch, Morris W and Pugh, Charles Chapman and Shub, Michael},
    title={Invariant manifolds},
    publisher = {Springer Berlin, Heidelberg},
    year = {1977}
}

@inproceedings{ansov1969geodesic,
  title={Geodesic flows on closed Riemannian manifolds with negative curvature},
  author={Ansov, DV},
  booktitle={Proc. Steklov Inst. Math.},
  volume={90},
  pages={1--235},
  year={1969}
}

@article{lind1982dynamical,
  title={Dynamical properties of quasihyperbolic toral automorphisms},
  author={Lind, Douglas A},
  journal={Ergodic Theory and Dynamical Systems},
  volume={2},
  number={1},
  pages={49--68},
  year={1982},
  publisher={Cambridge University Press}
}

@article{katzenlson1971ergodic,
  title={Ergodic automorphisms of T n are Bernoulli shifts},
  author={Katzenlson, Yitzhak},
  journal={Israel Journal of Mathematics},
  volume={10},
  number={2},
  pages={186--195},
  year={1971},
  publisher={Springer}
}

@book{adler1970similarity,
  title={Similarity of automorphisms of the torus},
  author={Adler, Roy L and Weiss, Benjamin},
  volume={98},
  year={1970},
  publisher={American Mathematical Society Providence, RI}
}

@article{adler1967entropy,
  title={Entropy, a complete metric invariant for automorphisms of the torus},
  author={Adler, Roy L and Weiss, Benjamin},
  journal={Proceedings of the National Academy of Sciences},
  volume={57},
  number={6},
  pages={1573--1576},
  year={1967}
}

@article{le1999limit,
  title={Limit theorems for non-hyperbolic automorphisms of the torus},
  author={Le Borgne, St{\'e}phane},
  journal={Israel Journal of Mathematics},
  volume={109},
  number={1},
  pages={61--73},
  year={1999},
  publisher={Springer}
}

@article{dedecker2012empirical,
  title={Empirical central limit theorems for ergodic automorphisms of the torus},
  author={Dedecker, J{\'e}r{\^o}me and Merlev{\`e}de, Florence and P{\`e}ne, Fran{\c{c}}oise},
  journal={arXiv preprint arXiv:1210.3546},
  year={2012}
}

@article {MR4198639,
    AUTHOR = {Avila, A. and Crovisier, S. and Wilkinson, A.},
     TITLE = {{$C^1$} density of stable ergodicity},
   JOURNAL = {Adv. Math.},
  FJOURNAL = {Advances in Mathematics},
    VOLUME = {379},
      YEAR = {2021},
     PAGES = {Paper No. 107496, 68},
      ISSN = {0001-8708,1090-2082},
   MRCLASS = {37C20 (37C40)},
  MRNUMBER = {4198639},
MRREVIEWER = {Luciana\ Silva\ Salgado},
       DOI = {10.1016/j.aim.2020.107496},
       URL = {https://doi.org/10.1016/j.aim.2020.107496},
}

@article {MR1298715,
    AUTHOR = {Grayson, Matthew and Pugh, Charles and Shub, Michael},
     TITLE = {Stably ergodic diffeomorphisms},
   JOURNAL = {Ann. of Math. (2)},
  FJOURNAL = {Annals of Mathematics. Second Series},
    VOLUME = {140},
      YEAR = {1994},
    NUMBER = {2},
     PAGES = {295--329},
      ISSN = {0003-486X,1939-8980},
   MRCLASS = {58F11 (58F17 58F30)},
  MRNUMBER = {1298715},
MRREVIEWER = {Remo\ Badii},
       DOI = {10.2307/2118602},
       URL = {https://doi.org/10.2307/2118602},
}

@article {MR1449765,
    AUTHOR = {Pugh, Charles and Shub, Michael},
     TITLE = {Stably ergodic dynamical systems and partial hyperbolicity},
   JOURNAL = {J. Complexity},
  FJOURNAL = {Journal of Complexity},
    VOLUME = {13},
      YEAR = {1997},
    NUMBER = {1},
     PAGES = {125--179},
      ISSN = {0885-064X,1090-2708},
   MRCLASS = {58F11 (58F15)},
  MRNUMBER = {1449765},
MRREVIEWER = {Christian\ Bonatti},
       DOI = {10.1006/jcom.1997.0437},
       URL = {https://doi.org/10.1006/jcom.1997.0437},
}

@article {MR1464,
    AUTHOR = {Hopf, Eberhard},
     TITLE = {Statistik der geod\"atischen {L}inien in {M}annigfaltigkeiten
              negativer {K}r\"ummung},
   JOURNAL = {Ber. Verh. S\"achs. Akad. Wiss. Leipzig Math.-Phys. Kl.},
  FJOURNAL = {Berichte \"uber die Verhandlungen der S\"achsischen Akademie
              der Wissenschaften zu Leipzig. Mathematisch-Physische Klasse},
    VOLUME = {91},
      YEAR = {1939},
     PAGES = {261--304},
      ISSN = {0366-0036},
   MRCLASS = {46.3X},
  MRNUMBER = {1464},
MRREVIEWER = {Gustav\ A.\ Hedlund},
}

@article {MR2018925,
    AUTHOR = {Bonatti, C. and D\'iaz, L. J. and Pujals, E. R.},
     TITLE = {A {$C^1$}-generic dichotomy for diffeomorphisms: weak forms of
              hyperbolicity or infinitely many sinks or sources},
   JOURNAL = {Ann. of Math. (2)},
  FJOURNAL = {Annals of Mathematics. Second Series},
    VOLUME = {158},
      YEAR = {2003},
    NUMBER = {2},
     PAGES = {355--418},
      ISSN = {0003-486X,1939-8980},
   MRCLASS = {37D30 (37C20)},
  MRNUMBER = {2018925},
MRREVIEWER = {Maria\ Jos\'e\ Pacifico},
       DOI = {10.4007/annals.2003.158.355},
       URL = {https://doi.org/10.4007/annals.2003.158.355},
}

@article {MR2085722,
    AUTHOR = {Tahzibi, Ali},
     TITLE = {Stably ergodic diffeomorphisms which are not partially
              hyperbolic},
   JOURNAL = {Israel J. Math.},
  FJOURNAL = {Israel Journal of Mathematics},
    VOLUME = {142},
      YEAR = {2004},
     PAGES = {315--344},
      ISSN = {0021-2172,1565-8511},
   MRCLASS = {37C40 (37C05 37C20 37D25)},
  MRNUMBER = {2085722},
MRREVIEWER = {Viorel\ Ni\c tic\u a},
       DOI = {10.1007/BF02771539},
       URL = {https://doi.org/10.1007/BF02771539},
}

@incollection {MR2039999,
    AUTHOR = {Dolgopyat, Dmitry and Wilkinson, Amie},
     TITLE = {Stable accessibility is {$C^1$} dense},
      NOTE = {Geometric methods in dynamics. II},
   JOURNAL = {Ast\'erisque},
  FJOURNAL = {Ast\'erisque},
    NUMBER = {287},
      YEAR = {2003},
     PAGES = {xvii, 33--60},
      ISSN = {0303-1179,2492-5926},
   MRCLASS = {37D30 (37C20 37J10)},
  MRNUMBER = {2039999},
MRREVIEWER = {Lorenzo\ J.\ D\'iaz},
}

@incollection {MR4142463,
    AUTHOR = {Avila, Artur and Viana, Marcelo},
     TITLE = {Stable accessibility with 2-dimensional center},
      NOTE = {Some aspects of the theory of dynamical systems: a tribute to
              Jean-Christophe Yoccoz. Vol. II},
   JOURNAL = {Ast\'erisque},
  FJOURNAL = {Ast\'erisque},
    NUMBER = {416},
      YEAR = {2020},
     PAGES = {301--320},
      ISSN = {0303-1179,2492-5926},
      ISBN = {978-2-85629-917-3},
   MRCLASS = {37D30 (37C86)},
  MRNUMBER = {4142463},
MRREVIEWER = {Jana\ Rodriguez\ Hertz},
       DOI = {10.24033/ast},
       URL = {https://doi.org/10.24033/ast},
}

@article{DIDIER_2003, title={Stability of accessibility}, volume={23}, DOI={10.1017/S0143385702001785}, number={6}, journal={Ergodic Theory and Dynamical Systems}, author={DIDIER, Ph.}, year={2003}, pages={1717–1731}}

@book {MR2391330,
     TITLE = {Partially hyperbolic dynamics, laminations, and
              {T}eichm\"uller flow},
    SERIES = {Fields Institute Communications},
    VOLUME = {51},
    EDITOR = {Forni, Giovanni and Lyubich, Mikhail and Pugh, Charles and
              Shub, Michael},
      NOTE = {Papers from the workshop held in Toronto, ON, January 2006},
 PUBLISHER = {American Mathematical Society, Providence, RI; Fields
              Institute for Research in Mathematical Sciences, Toronto, ON},
      YEAR = {2007},
     PAGES = {x+339},
      ISBN = {978-0-8218-4274-4},
   MRCLASS = {37-06 (30-06 32-06 37D30 37D40 57-06)},
  MRNUMBER = {2391330},
       DOI = {10.1090/fic/051},
       URL = {https://doi.org/10.1090/fic/051},
}

@article {MR538680,
    AUTHOR = {Herman, Michael-Robert},
     TITLE = {Sur la conjugaison diff\'erentiable des diff\'eomorphismes du
              cercle \`a{} des rotations},
   JOURNAL = {Inst. Hautes \'Etudes Sci. Publ. Math.},
  FJOURNAL = {Institut des Hautes \'Etudes Scientifiques. Publications
              Math\'ematiques},
    NUMBER = {49},
      YEAR = {1979},
     PAGES = {5--233},
      ISSN = {0073-8301,1618-1913},
   MRCLASS = {58F11 (28D99)},
  MRNUMBER = {538680},
MRREVIEWER = {C.\ S.\ Hartzman},
       URL = {http://www.numdam.org/item?id=PMIHES_1979__49__5_0},
}

@article {MR4422214,
    AUTHOR = {Leguil, Martin and Zhang, Zhiyuan},
     TITLE = {{$C^r$}-prevalence of stable ergodicity for a class of
              partially hyperbolic systems},
   JOURNAL = {J. Eur. Math. Soc. (JEMS)},
  FJOURNAL = {Journal of the European Mathematical Society (JEMS)},
    VOLUME = {24},
      YEAR = {2022},
    NUMBER = {9},
     PAGES = {3379--3438},
      ISSN = {1435-9855,1435-9863},
   MRCLASS = {37D30 (37C20)},
  MRNUMBER = {4422214},
MRREVIEWER = {Xiaodong\ Wang},
       DOI = {10.4171/jems/1163},
       URL = {https://doi.org/10.4171/jems/1163},
}

@article{Adam_2017,
doi = {10.1088/1361-6544/aa59a9},
url = {https://doi.org/10.1088/1361-6544/aa59a9},
year = {2017},
month = {feb},
publisher = {IOP Publishing},
volume = {30},
number = {3},
pages = {1146},
author = {Adam, Alexander},
title = {Generic non-trivial resonances for Anosov diffeomorphisms},
journal = {Nonlinearity}
}

@article{micena2021lyapunov,
  title={Lyapunov exponents everywhere and rigidity},
  author={Micena, Fernando Pereira and Llave, Rafael de la},
  journal={Journal of Dynamical and Control Systems},
  volume={27},
  number={4},
  pages={819--831},
  year={2021},
  publisher={Springer}
}

@article{gogolev2025smooth,
  title={Smooth rigidity for very non-algebraic Anosov diffeomorphisms of codimension one},
  author={Gogolev, Andrey and Hertz, Federico Rodriguez},
  journal={Israel Journal of Mathematics},
  volume={269},
  number={2},
  pages={801--852},
  year={2025},
  publisher={Springer}
}

@article{brown2026lyapunov,
  title={Lyapunov spectrum rigidity and simultaneous linearization for random Anosov diffeomorphisms},
  author={Brown, Aaron and Shi, Yi},
  journal={arXiv preprint arXiv:2601.04679},
  year={2026}
}

@article {MR2736152,
    AUTHOR = {Avila, Artur},
     TITLE = {On the regularization of conservative maps},
   JOURNAL = {Acta Math.},
  FJOURNAL = {Acta Mathematica},
    VOLUME = {205},
      YEAR = {2010},
    NUMBER = {1},
     PAGES = {5--18},
      ISSN = {0001-5962,1871-2509},
   MRCLASS = {58C25 (37C20)},
  MRNUMBER = {2736152},
MRREVIEWER = {Jairo\ Bochi},
       DOI = {10.1007/s11511-010-0050-y},
       URL = {https://doi.org/10.1007/s11511-010-0050-y},
}

@book{pesin2004lectures,
  title={Lectures on partial hyperbolicity and stable ergodicity},
  author={Pesin, Ya B},
  volume={1},
  year={2004},
  publisher={European Mathematical Society}
}

\end{document}